\documentclass[final]{opt2024} 

\DeclareMathOperator*{\argmin}{arg\,min}
\DeclareMathOperator*{\rank}{rank}
\DeclareMathOperator*{\Span}{span}

\newcommand{\sKHat}{\hat{s}_k}
\newcommand{\xK}{x_k}
\newcommand{\xKPlusOne}{x_{k+1}}
\newcommand{\fK}{f(x_k)}
\newcommand{\fKPlusOne}{f(x_k + s_k)}
\newcommand{\sK}{s_k}
\newcommand{\kappaS}{\kappa_S}
\newcommand{\kappaT}{\kappa_T}
\newcommand{\mKHat}[1]{\hat{m}_k\left( #1\right)}
\newcommand{\normTwo}[1]{\left\lVert#1\right\rVert_2}
\newcommand{\hessFK}{\nabla^2 f(x_k)}
\newcommand{\SK}{S_k}
\newcommand{\R}{\mathbb{R}}
\newcommand{\texteq}[1]{\text{\quad #1}}

\usepackage{float}
\usepackage{booktabs}

\title[Scalable Second-Order Optimization Algorithms for Minimizing Low-rank Functions]{Scalable Second-Order Optimization Algorithms for Minimizing Low-rank Functions}

\optauthor{%
\Name{Edward Tansley} \Email{tansley@maths.ox.ac.uk}\\
\Name{Coralia Cartis} \Email{cartis@maths.ox.ac.uk}\\
\addr Mathematical Institute, University of Oxford}

\begin{document}

\maketitle

\begin{abstract}%
We present a random-subspace variant of cubic regularization algorithm that chooses the size of the subspace adaptively, based on the rank of the projected second derivative matrix. Iteratively, our variant only requires access to  (small-dimensional) projections of first- and second-order problem derivatives  and calculates a reduced step inexpensively. The ensuing method maintains the optimal global rate of convergence of  (full-dimensional) cubic regularization, while showing improved scalability both theoretically and numerically, particularly when applied to low-rank functions. When applied to the latter, our algorithm naturally adapts the subspace size to the true rank of the function, without knowing it a priori. 
\end{abstract}

\section{Introduction}

Second-order optimization algorithms for the unconstrained optimization problem
\begin{equation*}
\min_{x\in \R^d} f(x),
\end{equation*}
where $f:\R^d\rightarrow \R$ is a sufficiently smooth, bounded-below function, 
use gradient and curvature information to determine iterates and so often experience faster convergence than first-order algorithms that only rely on gradient information. However, for high-dimensional problems, the computational complexity of these methods can be a barrier to their use in practice. We are concerned with the task of scaling up second-order optimization algorithms so that they are a practical option for high-dimensional problems.

A second-order algorithm designed to cope with high-dimensional problems is the \mbox{R-ARC} algorithm \cite{Zhen-PhD, shaoRandomsubspaceAdaptiveCubic2022}, a random subspace variant of the Adaptive Regularization using Cubics (ARC) algorithm \cite{Cartis:2009fq}. Subject to certain conditions on the random subspaces, R-ARC can attain the same convergence rate to an $\epsilon$-approximate first-order minimizer as ARC. These conditions imply that R-ARC is particularly effective for functions with Hessians of rank bounded by some $r$ (significantly) lower than the function dimension $d$. A class of functions with this property are \textit{low-rank functions} \cite{wang_bayesian_2016}, which have been frequently studied in the context of machine learning.

\paragraph{ARC and R-ARC}
ARC \cite{nesterovCubicRegularizationNewton2006, Cartis:2009fq} is an iterative algorithm that at iteration $k$, determines the step $s_k$ by (approximately) solving the following local model\footnote{Here $\sigma_{k}/3$ replaces $L/6$ in the well-known bound $f(x_{k} + s) \leq f(x_{k}) + \langle \nabla f(x_{k}),\ s\rangle +  \frac{1}{2}\langle s,\ \nabla^{2} f(x_{k})s\rangle + \frac{\sigma_{k}}{3}\|s\|_{2}^{3}$ (where $L$ is the Lipschitz constant of $\nabla^{2} f$).}:
    \begin{equation*}
        \argmin_{s \in \R^{d}} m_{k}(s) = f(x_{k}) + \langle \nabla f(x_{k}),\ s\rangle +  \frac{1}{2}\langle s,\ \nabla^{2} f(x_{k})s\rangle + \frac{\sigma_{k}}{3}\|s\|_{2}^{3}
    \end{equation*}
where $\nabla f$ and $\nabla^{2} f$ denote the gradient and Hessian of $f$, $x_{k}$ is the current iterate and $\sigma_{k}$ is the regularization parameter. Assuming Lipschitz continuity of the Hessian on the iterates' path, ARC requires at most $\mathcal{O}(\epsilon^{-3/2})$ iterations to attain an $\epsilon$-approximate first-order minimizer; this convergence rate is optimal over a large class of second-order methods \cite{cartis_evaluation_2022}.

A random subspace variant of ARC, R-ARC was introduced in \cite{Zhen-PhD, shaoRandomsubspaceAdaptiveCubic2022}. At each iteration $k$, a random sketching matrix $S_{k} \in \R^{l \times d}$ is drawn from a distribution $\mathcal{S}$ and the search space for $\sK \in \R^{d}$ is restricted so the $l$-dimensional subspace $\Span(\SK^{\top})$. These papers prove that assuming certain embedding conditions on $\mathcal{S}$, the optimal $\mathcal{O}(\epsilon^{-3/2})$ iteration complexity can be attained by ARC, despite only accessing projected (first- and second-order) problem information at each iteration. We now state an informal version of this result, restricted to Gaussian matrices.

\begin{theorem}[Informal, \cite{Zhen-PhD, shaoRandomsubspaceAdaptiveCubic2022}]
    Suppose that $\mathcal{S}$ is the distribution of (scaled) $l \times d$ Gaussian matrices with $l = \mathcal{O}(r+1)$, where $r \leq d$ is an upper bound on the maximum rank of $\,\nabla^{2}f(x_k)$ across all iterations, and that $f$ has globally Lipschitz continuous second derivatives. Then R-ARC achieves the optimal $\mathcal{O}(\epsilon^{-3/2})$ rate of convergence, with high probability.
    \label{thm:R-ARC_convergence}
\end{theorem}

The proof of this Theorem relies upon $\mathcal{S}$ being an oblivious subspace embedding (Definition \ref{def:oblivious_embedding}) for  matrices with rank $r+1$. Similar results can be established for matrix distributions other than Gaussian, with $l$ possibly having a different dependency on $r$ \cite{cartis_randomised_2022}. Theorem \ref{thm:R-ARC_convergence} can be applied to any suitable objective function $f$. However, the requirement that $l = \mathcal{O}(r+1)$ means that unless $r \ll d$, R-ARC is not guaranteed to be able  to  gain a significant dimensionality over ARC (by using only little problem information and computing an inexpensive reduced step), whilst maintaining the $\mathcal{O}(\epsilon^{-3/2})$ convergence rate.

In R-ARC, the sketch dimension $l$ is fixed throughout the run of the algorithm. Theorem \ref{thm:R-ARC_convergence} requires $l$ to be proportional to a bound $r$ on the maximal Hessian rank at the iterates, but this may not be known a priori. This motivates us  to develop a variant of R-ARC that can adapt the sketch/subspace size  to local problem information.

\paragraph{Low-rank functions}
We now define a class of functions that particularly benefit from random subspace algorithms. These functions are also known as functions with \textit{low effective dimensionality}, with \textit{active subspaces} or \textit{multi-ridge} functions \cite{cartisDimensionalityReductionTechnique2022, cartis_learning_2024}.

\begin{definition}[Low-rank Functions \cite{wang_bayesian_2016}]\label{def:low:rank}
    A function $f:\R^{d} \rightarrow \R$ is said to be of rank $r$, with $r \leq d$ if
    \begin{itemize}
        \item there exists a linear subspace $\mathcal{T}$ of dimension $r$ such that for all $x_{\top} \in \mathcal{T} \subset \R^{d}$ and $x_{\perp} \in \mathcal{T}^{\perp} \subset \R^{d}$, we have $f(x_{\top} + x_{\perp}) = f(x_{\top})$, where $\mathcal{T}^{\perp}$ is the orthogonal complement of $\mathcal{T}$;
        \item $r$ is the smallest integer with this property
    \end{itemize}
    We call $\mathcal{T}$ the effective subspace of $f$ and $\mathcal{T}^{\perp}$ the constant subspace of $f$. 
\end{definition}
We state a lemma whose proof is included in Appendix \ref{sec:results} and which then helps us apply the results in \cite{Zhen-PhD, shaoRandomsubspaceAdaptiveCubic2022} to low-rank functions (note the requirement on a bound on the $\rank(\hessFK)$ in Theorem \ref{thm:R-ARC_convergence}).

\begin{lemma}\label{lem:low:rank:hessian}
   If $f:\R^{d} \rightarrow \R$ is a low-rank function of rank $r$, and $f$ is $C^2$, then for all $x \in \R^{d}$, $\nabla^{2} f(x)$ has rank at most $r$.
\end{lemma}
Overparameterized models in various applications are candidates for low-rank behaviour as we expect invariance to some reparameterization; provided such invariance is (approximately) linear. In deep neural networks, there are two sources of low-rank behaviour: the training loss as a function of parameters \cite{cosson_gradient_2022}, and the trained net as a function of the input data \cite{parkinsonReLUNeuralNetworks2024}. Further,  in hyperparameter optimization, low-rank behaviour is observed as network performance only depends upon a selection of hyperparameters \cite{bergstraRandomSearchHyperParameter2012}.

\paragraph{Contributions} We introduce R-ARC-D, a new variant of the R-ARC algorithm that can vary the size of the random susbspace between iterations. We detail an update scheme for the sketch size that adapts to the local Hessian rank of the iterates. This algorithm attains the optimal $\mathcal{O}(\epsilon^{-3/2})$ iteration complexity for an $\epsilon$-approximate first-order minimizer, whilst maintaining a sketch dimension $l_{k}$ that is $\mathcal{O}(r)$ for functions with Hessians with rank bounded by $r$, in particular low-rank functions, despite the algorithm not needing to know $r$ a priori. Through numerical experiments on low-rank problems, we demonstrate the superior efficiency of this algorithm compared to the R-ARC and ARC algorithms on these problems.

\section{Algorithm and Main Results}\label{sec:algorithm}

In this section, we present the R-ARC-D algorithm and conditions under which it can attain the optimal $\mathcal{O}(\epsilon^{-3/2})$ iteration complexity for finding an $\epsilon$-approximate first-order local minimizer. \mbox{R-ARC-D} differs from the R-ARC algorithm presented by \cite{Zhen-PhD, shaoRandomsubspaceAdaptiveCubic2022} as it allows for the sketch size to be iteratively adjusted.

{\small{\begin{algorithm2e}[H]
    \caption{Random subspace cubic regularisation algorithm with variable sketching\\ dimension \mbox{(R-ARC-D)}}
    \SetAlgoLined

     Choose constants $\theta \in (0,1)$, $\kappaT, \kappaS \geq 0$. Initialize the algorithm by setting 
 $x_0 \in \R^d$ and $k=0$.\\

    Draw a random matrix $S_k \in \R^{l_k \times d}$ from $\cal{S}$, and let
    \begin{equation}\label{m_k_hat}
        \hat{m}_{k}(\hat{s}) = \underbrace{f(x_k) + \langle \hat{\nabla} f(x_k),\ \hat{s}\rangle +  \frac{1}{2}\langle \hat{s},\ \hat{\nabla}^{2} f(x_{k})\hat{s}\rangle}_{\displaystyle \hat{q}_{k}(\hat{s})} + \frac{\sigma_{k}}{3}\|S_{k}^{\top} \hat{s}\|_{2}^{3},
    \end{equation}
    where $\hat{\nabla} f(x_{k}) = S_{k}\nabla f(x_{k})$ and $\hat{\nabla}^{2} f(x_{k}) = S_{k}\nabla^{2}f(x_{k}) S_{k}^{\top}$.\\
    \vspace{1em}
    
    Compute $\sKHat\in \R^{l_k}$ by approximately minimizing $\hat{m}$ such that
    \begin{align*}
        \hat{m}(\sKHat) \leq \hat{m}(0); \quad
        \normTwo{\nabla \hat{m}(\sKHat)} \leq 
        \kappa_T \|S_{k}^{\top} \hat{s}\|_{2}^{2}; \quad
        \nabla^2 \mKHat{\sKHat} \succeq -\kappaS \|S_{k}^{\top} \hat{s}\|_{2}.
    \end{align*}
    
    Compute a trial step
        $\sK = \SK^T \sKHat.$
    
    Compute and check whether the decrease ratio satisfies:
    \begin{equation*}
     \rho_{k}: = \frac{\fK - \fKPlusOne}{\fK - \hat{q}_{k}(\hat{s})}\geq \theta.
    \end{equation*}
    
    If the decrease condition holds, set $\xKPlusOne = \xK + \sK$ and $\sigma_{k+1} < \sigma_{k}$ [successful iteration]. 
    
    Otherwise set $\xKPlusOne = \xK$ and $\sigma_{k+1} > \sigma_{k}$ [unsuccessful iteration].
    
    Increase the iteration count by setting $k=k+1$ in both cases. Set $l_{k+1} \geq l_k$
    \label{alg:R-ARC-D}
    
\end{algorithm2e}}}

R-ARC as presented in \cite{Zhen-PhD, shaoRandomsubspaceAdaptiveCubic2022} can be recovered from Algorithm \ref{alg:R-ARC-D} by fixing $l_k = l$ from some $l \geq 1$. Further setting $S_k = I_d$ for each iteration recovers the original ARC algorithm.

\paragraph{Evaluating sketched problem information} Algorithm \ref{alg:R-ARC-D}   only requires projected objective's gradients and Hessians (see $\hat{\nabla} f(x_{k})$ and $\hat{\nabla}^{2} f(x_{k})$ in \eqref{m_k_hat}).  These can be calculated efficiently, without evaluation of full gradients and Hessians, using techniques such as using directional derivatives, block finite differences or automatic differentiation. For example, $\hat{\nabla} f(x_{k})$ requires only $l_k$ directional derivatives of $f$ with respect  to the rows of $S_k$.

\subsection{An adaptive sketch size rule}

We introduce a sketch size update rule that can be included in Algorithm \ref{alg:R-ARC-D}. The motivation behind this update rule is that we seek to use local problem information to, in a sense, learn the rank of the function $f$ (assuming that it has low-rank structure). To do this, we keep track of the observed ranks of the sketched Hessian. In Algorithm \ref{alg:R-ARC-D}, we define the following for $k \geq 0:$
\begin{equation*}
    r_k := \rank(\hessFK);\quad \hat{r}_k := \rank(\SK \hessFK \SK^T);\quad \hat{R}_k := \max_{1 \leq j \leq k}\hat{r}_k.
\end{equation*}
Using these, we can give the following update rule:

\begin{equation}\label{eq:l_k:update:step}
    l_{k+1} = \begin{cases}
        \max( C\hat{R}_{k} + 1, l_{k})&\text{if $\hat{R}_k > \hat{R}_{k-1}$} \\
        l_k &\text{otherwise.}
    \end{cases}
    \tag{$\star$}
\end{equation}
where $C \geq 1$ is a user-defined constant. We make two remarks:
\begin{enumerate}
    \item in the case that $C = 1$, we simply need to assess whether the sketched Hessian is singular, rather than know its rank.
    This is because $\hat{R}_k \leq l_k$ so $l_{k+1} > l_k$ only if $\hat{r}_{k} = l_k$.
    \item for all $k$, we have $l_{k} \leq \max(Cr + 1, l_0)$ where $r$ is the rank of $f$, and hence $l_{k}$ remains $\mathcal{O}(r)$.
\end{enumerate}
We now seek to show that if this update rule is followed, the sketch dimension $l_k$ will increase in a manner that enables the same $\mathcal{O}(\epsilon^{-3/2})$ iteration complexity to drive the norm of the objective's gradient norm below $\epsilon$ as Theorem \ref{thm:R-ARC_convergence}.

\begin{lemma}\label{lem:rank_preserve}
    Letting $S_{k} \in \R^{l_k \times d}$ be a Gaussian matrix with $l_{k} \leq d$, we have
    \begin{equation}
        \mathbb{P}(\hat{r}_{k} = \min(l_{k}, r_{k})) = 1.
    \end{equation}
\end{lemma}
We apply this Lemma to prove the following Corollary.

\begin{lemma}\label{lem:lk:increases}
    Set $l_0 \geq 1$ and suppose that the update rule (\ref{eq:l_k:update:step}) is applied to $l_k$. For all $k \geq 1$, If $l_{k} < Cr_{k} + 1$, then with probability 1, $\hat{R}_{k} > \hat{R}_{k-1}$.
\end{lemma}
\begin{proof}
    By the update rule \eqref{eq:l_k:update:step}, we have that $l_{k} < Cr_{k} + 1 \implies \hat{R}_{k-1} < r_{k}$. We also have that $l_{k} \geq \hat{R}_{k-1} + 1$. Therefore, we have that $l_{k} < Cr_{k} + 1 \implies \hat{R}_{k-1} < \min(l_{k}, r_{k})$.  Hence, by Lemma \ref{lem:rank_preserve}, we have that $\hat{r}_{k} = \min(l_{k}, r_{k}) > \hat{R}_{k-1}$ with probability 1.
\end{proof}
Applying this Lemma allows us to prove the following convergence result.

\begin{theorem}
    [R-ARC-D convergence result]
    Suppose that $\mathcal{S}$ is the distribution of scaled Gaussian matrices and $f$ is a low-rank function of rank $r$ with Lipschitz-continuous second derivatives. Apply Algorithm \ref{alg:R-ARC-D} with the sketch update rule (\ref{eq:l_k:update:step}) with $l_0 \geq 1, C =  \lceil 4C_{l}(2+ \log(16)) \rceil$ where $C_{l}$ is defined in Lemma \ref{lem:Gauss_embedding}, then R-ARC achieves the optimal $\mathcal{O}(\epsilon^{-3/2})$ rate of convergence, with high probability.\\
    \label{thm:R-ARC-D_convergence}
\end{theorem}
In terms of dimension-dependence, the $\mathcal{O}$ bound on the number of iterations in  Theorem \ref{thm:R-ARC-D_convergence} is proportional to $\sqrt{d/r}$ as long as $r_k \leq r$. Thus R-ARC-D benefits from the same optimal convergence rate result as R-ARC and ARC, up to a constant. Furthermore, when applied to low-rank functions, the algorithm is able to learn the function rank whilst solving local subproblems in smaller-dimensional subspaces.

\section{Numerical Experiments}

In these numerical experiments, we apply the R-ARC-D algorithm as described in Algorithm \ref{alg:R-ARC-D}, using the $l_k$ update rule \eqref{eq:l_k:update:step} with $C = 1$ for simplicity. The code we used is a modification of the  ARC code used in \cite{Cartis:2009fq}. We make a minor modification in that we only redraw $S_{k}$ after successful iterations; this update step performs better empirically than redrawing after each iteration. The performance of R-ARC-D is compared with that of R-ARC and ARC. As a measure of budget, we use relative Hessians seen; if at iteration $k$, we draw a sketching matrix of size $l_k \times d$, we see $(l_k / d)^{2}$ relative Hessians. When calculating $\hat{\nabla} f(x_{k})$ and $\hat{\nabla}^{2} f(x_{k})$, we calculate $\nabla f(x_{k})$ and $\nabla^{2} f(x_{k})$, and then multiply by $S_{k}$; the computational efficiency could be much  improved by applying techniques discussed in Section \ref{sec:algorithm}.

\paragraph{Augmented CUTEst problems}

To create low-rank functions to test on, we take CUTEst \cite{gould2015cutest} problems of dimension (or rank) $r \approx 100$ and add dimensions and rotate to create low-rank problems of dimension $d = 1000$. Given a function $f: \R^{r} \rightarrow \R$, this can be achieved by sampling a random orthogonal matrix $Q \in \R^{d \times r}$ so that $Q^{\top}Q = I_{r}$ and define $g: \R^{d} \rightarrow \R$ by $g(x) = f(Q^{\top}x)$ to be a low-rank variant of $f$. To distinguish these problems from the standard CUTEst versions, we prefix the problem names with ``l-". Problem details can be found in Table \ref{tab:cutest_lowrank}.

\paragraph{R-ARC-D update step} In Figure \ref{fig:R_ARC_D_update_lARTIF}, we plot a test to demonstrate how the adaptive sketch update rule \eqref{eq:l_k:update:step} works on an example problem, starting from $l_0 = 2$.

\begin{figure}[H]
    \centering
    \includegraphics[width=0.48\linewidth]{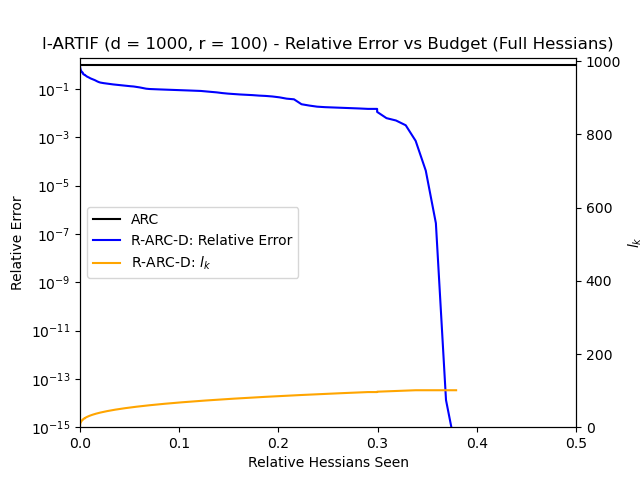}
    \includegraphics[width=0.48\linewidth]{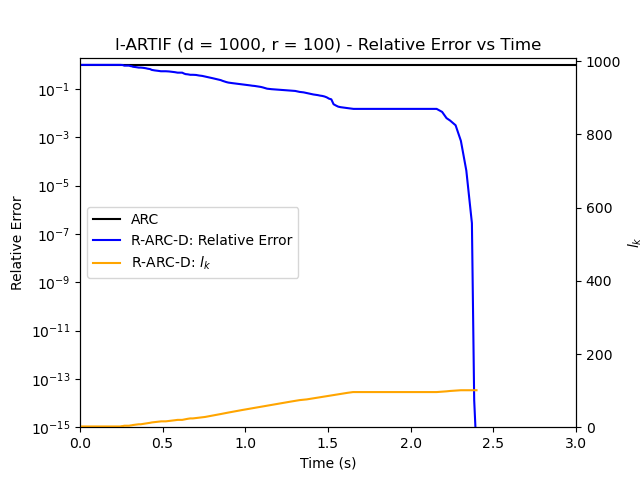}
    \caption{Example of R-ARC-D applied to the low-rank problem l-ARTIF}
    \label{fig:R_ARC_D_update_lARTIF}
\end{figure}
We see that the sketch dimension $l_k$ increases to eventually reach the function rank. For l-ARTIF, we see that R-ARC-D significantly outperforms ARC, which fails to make a step in the budget and time taken for R-ARC-D to converge. We can also apply R-ARC-D to full-rank problems such as ARTIF (with parameter N = 1000), which we plot in Figure \ref{fig:R_ARC_D_update_ARTIF}, where R-ARC-D converges, but not faster than ARC.

\paragraph{Data Profiles}

Here we compare the performance between R-ARC-D and R-ARC through data profiles \cite{moreBenchmarkingDerivativeFreeOptimization2009}. A description of the methodology can be found in Appendix \ref{sec:data_profiles}. The set of problems considers can be found in Table \ref{tab:cutest_lowrank}. For R-ARC-D, we again set $l_0 = 2$, whilst for R-ARC, we sketch at $1\%,\ 5\%\text{ and } 7.5\%$ of the original problem dimension. As the functions are of dimension $d = 1000$ with rank $r \approx 100$, this corresponds to $\approx 10-75\%$ of the function rank. The results are plotted in Figure \ref{fig:data_profiles}, where we show results for tolerances $\tau = 1e-2$ (low-precision) and $\tau = 1e-5$ (high precision). We repeat each problem $5$ times, with different $Q$ matrices, treating each run as a separate problem in the plots.

\begin{figure}[H]
    \centering
    \includegraphics[width=0.48\linewidth]{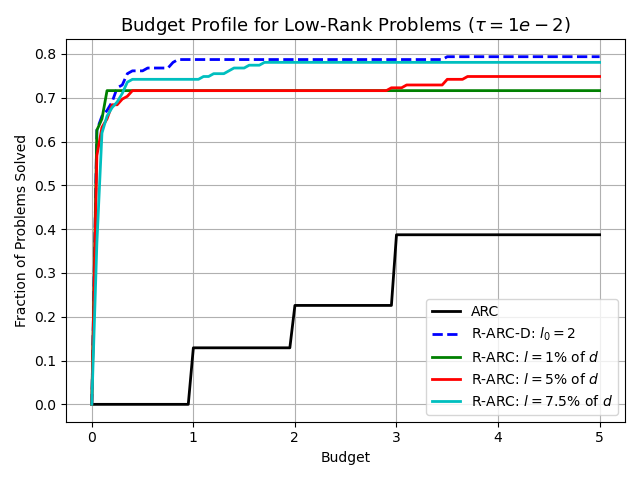}
    \includegraphics[width=0.48\linewidth]{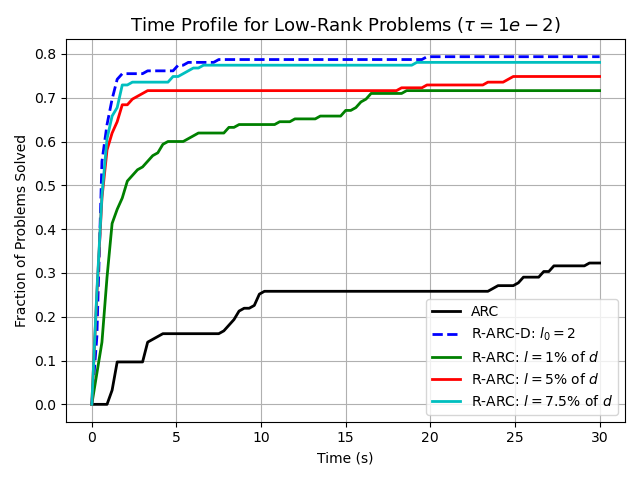}
    \includegraphics[width=0.48\linewidth]{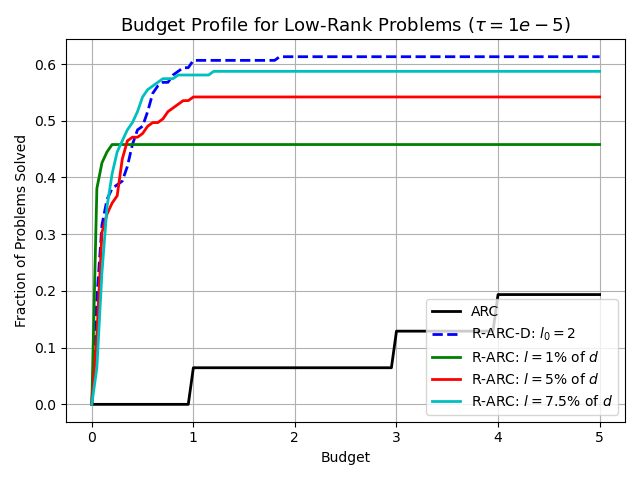}
    \includegraphics[width=0.48\linewidth]{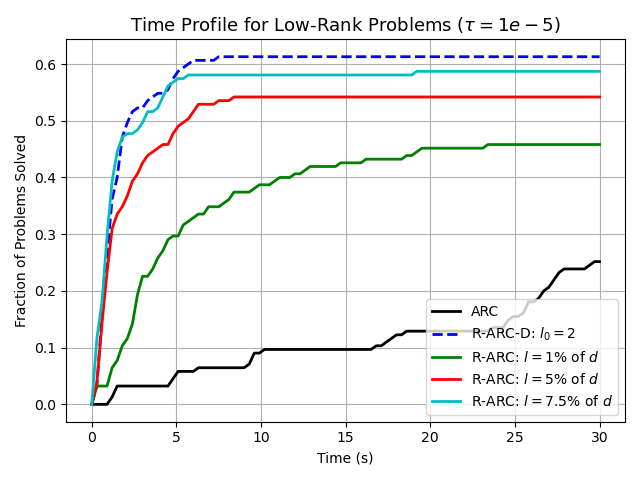}
    \label{fig:data_profiles}
    \caption{Data profiles of R-ARC-D compared to R-ARC and ARC}
\end{figure}

We see that R-ARC-D started at $l_0 = 2$ outperforms R-ARC regardless of the fixed sketch size used by R-ARC. This is more significant for the high-precision solutions ($\tau = 1e-5$) where R-ARC-D typically increases $l_k$ until it reaches the function rank $r$. In Figure \ref{fig:individual_problems}, we plot individual problem performance for several of the problems considered here. In these plots, we see that whilst R-ARC with $l = 1\%$ of $d$ occasionally performs best in terms of relative Hessians, it performs worst in terms of time due to it taking significantly more iterations. Thus overall, we found that R-ARC-D performs particularly well on low-rank problems, which we have demonstrated both numerically and theoretically.

\paragraph{Acknowledgments} The first author's (ET) research was supported by the Oxford Centre for Doctoral Training in Mathematics of Random Systems (EPSRC EP/S023925/1), while the second author's (CC), by the Hong Kong Innovation and Technology Commission's Center for Intelligent Multi-dimensional Analysis (InnoHK Project CIMDA).

\bibliography{combined}

\begin{thebibliography}{17}
\providecommand{\natexlab}[1]{#1}
\providecommand{\url}[1]{\texttt{#1}}
\expandafter\ifx\csname urlstyle\endcsname\relax
  \providecommand{\doi}[1]{doi: #1}\else
  \providecommand{\doi}{doi: \begingroup \urlstyle{rm}\Url}\fi

\bibitem[Bergstra and Bengio(2012)]{bergstraRandomSearchHyperParameter2012}
James Bergstra and Yoshua Bengio.
\newblock Random {{Search}} for {{Hyper-Parameter Optimization}}.
\newblock \emph{Journal of Machine Learning Research}, 13\penalty0 (10):\penalty0 281--305, 2012.
\newblock ISSN 1533-7928.

\bibitem[Cartis and Otemissov(2022)]{cartisDimensionalityReductionTechnique2022}
Coralia Cartis and Adilet Otemissov.
\newblock A dimensionality reduction technique for unconstrained global optimization of functions with low effective dimensionality.
\newblock \emph{Information and Inference: A Journal of the IMA}, 11\penalty0 (1):\penalty0 167--201, March 2022.
\newblock ISSN 2049-8772.
\newblock \doi{10.1093/imaiai/iaab011}.
\newblock URL \url{https://doi.org/10.1093/imaiai/iaab011}.

\bibitem[Cartis et~al.(2011)Cartis, Gould, and Toint]{Cartis:2009fq}
Coralia Cartis, Nicholas~IM Gould, and Philippe~L Toint.
\newblock Adaptive cubic regularisation methods for unconstrained optimization. part i: motivation, convergence and numerical results.
\newblock \emph{Mathematical Programming}, 127\penalty0 (2):\penalty0 245--295, 2011.

\bibitem[Cartis et~al.(2022)Cartis, Fowkes, and Shao]{cartis_randomised_2022}
Coralia Cartis, Jaroslav Fowkes, and Zhen Shao.
\newblock Randomised subspace methods for non-convex optimization, with applications to nonlinear least-squares, November 2022.
\newblock URL \url{http://arxiv.org/abs/2211.09873}.
\newblock arXiv:2211.09873 [math].

\bibitem[Cartis et~al.(2024)Cartis, Liang, Massart, and Otemissov]{cartis_learning_2024}
Coralia Cartis, Xinzhu Liang, Estelle Massart, and Adilet Otemissov.
\newblock Learning the subspace of variation for global optimization of functions with low effective dimension, January 2024.
\newblock URL \url{http://arxiv.org/abs/2401.17825}.
\newblock arXiv:2401.17825 [math].

\bibitem[Coralia~Cartis and Toint(2022)]{cartis_evaluation_2022}
Nicholas I.M.~Gould Coralia~Cartis and Philippe~L. Toint.
\newblock \emph{Evaluation complexity of algorithms for nonconvex optimization: theory, computation, and perspectives}.
\newblock Society for Industrial and Applied Mathematics, Philadelphia, 2022.
\newblock ISBN 978-1-61197-699-1.

\bibitem[Cosson et~al.(2022)Cosson, Jadbabaie, Makur, Reisizadeh, and Shah]{cosson_gradient_2022}
Romain Cosson, Ali Jadbabaie, Anuran Makur, Amirhossein Reisizadeh, and Devavrat Shah.
\newblock Gradient {Descent} for {Low}-{Rank} {Functions}, June 2022.
\newblock URL \url{http://arxiv.org/abs/2206.08257}.
\newblock arXiv:2206.08257 [cs, math].

\bibitem[Dolan and Mor{\'e}(2002)]{dolan2002benchmarking}
Elizabeth~D Dolan and Jorge~J Mor{\'e}.
\newblock Benchmarking optimization software with performance profiles.
\newblock \emph{Mathematical programming}, 91\penalty0 (2):\penalty0 201--213, 2002.

\bibitem[Gould et~al.(2015)Gould, Orban, and Toint]{gould2015cutest}
Nicholas~IM Gould, Dominique Orban, and Philippe~L Toint.
\newblock {CUTEst}: a constrained and unconstrained testing environment with safe threads for mathematical optimization.
\newblock \emph{Computational Optimization and Applications}, 60\penalty0 (3):\penalty0 545--557, 2015.

\bibitem[Mor{\'e} and Wild(2009)]{moreBenchmarkingDerivativeFreeOptimization2009}
Jorge~J. Mor{\'e} and Stefan~M. Wild.
\newblock Benchmarking {Derivative}-{Free} {Optimization} {Algorithms}.
\newblock \emph{SIAM Journal on Optimization}, 20\penalty0 (1):\penalty0 172--191, January 2009.
\newblock ISSN 1052-6234.
\newblock \doi{10.1137/080724083}.
\newblock URL \url{https://epubs.siam.org/doi/10.1137/080724083}.
\newblock Publisher: Society for Industrial and Applied Mathematics.

\bibitem[Nesterov and Polyak(2006)]{nesterovCubicRegularizationNewton2006}
Yurii Nesterov and B.T. Polyak.
\newblock Cubic regularization of {Newton} method and its global performance.
\newblock \emph{Mathematical Programming}, 108\penalty0 (1):\penalty0 177--205, August 2006.
\newblock ISSN 1436-4646.
\newblock \doi{10.1007/s10107-006-0706-8}.
\newblock URL \url{https://doi.org/10.1007/s10107-006-0706-8}.

\bibitem[Parkinson et~al.(2024)Parkinson, Ongie, and Willett]{parkinsonReLUNeuralNetworks2024}
Suzanna Parkinson, Greg Ongie, and Rebecca Willett.
\newblock {ReLU} {Neural} {Networks} with {Linear} {Layers} are {Biased} {Towards} {Single}- and {Multi}-{Index} {Models}, June 2024.
\newblock URL \url{http://arxiv.org/abs/2305.15598}.
\newblock arXiv:2305.15598 [cs, stat].

\bibitem[{Sarl\'os}(2006)]{10.1109/FOCS.2006.37}
Tam\'as {Sarl\'os}.
\newblock Improved approximation algorithms for large matrices via random projections.
\newblock In \emph{2006 47th Annual IEEE Symposium on Foundations of Computer Science (FOCS'06)}, pages 143--152, 2006.
\newblock \doi{10.1109/FOCS.2006.37}.

\bibitem[Shao(2022)]{Zhen-PhD}
Zhen Shao.
\newblock \emph{On Random Embeddings and Their Application to Optimisation}.
\newblock PhD thesis, Mathematical Institute, University of Oxford, 2022.

\bibitem[Shao and Cartis(2022)]{shaoRandomsubspaceAdaptiveCubic2022}
Zhen Shao and Coralia Cartis.
\newblock Random-subspace adaptive cubic regularisation method for nonconvex optimisation.
\newblock In \emph{{HOO}-2022: {Order} up! {The} {Benefits} of {Higher}-{Order} {Optimization} in {Machine} {Learning}}, 2022.

\bibitem[Wang et~al.(2016)Wang, Hutter, Zoghi, Matheson, and De~Feitas]{wang_bayesian_2016}
Ziyu Wang, Frank Hutter, Masrour Zoghi, David Matheson, and Nando De~Feitas.
\newblock Bayesian {Optimization} in a {Billion} {Dimensions} via {Random} {Embeddings}.
\newblock \emph{Journal of Artificial Intelligence Research}, 55:\penalty0 361--387, February 2016.
\newblock ISSN 1076-9757.
\newblock \doi{10.1613/jair.4806}.
\newblock URL \url{https://jair.org/index.php/jair/article/view/10983}.

\bibitem[Woodruff(2014)]{10.1561/0400000060}
David~P. Woodruff.
\newblock Sketching as a tool for numerical linear algebra.
\newblock \emph{Found. Trends Theor. Comput. Sci.}, 10\penalty0 (1-2):\penalty0 1--157, 2014.
\newblock ISSN 1551-305X.
\newblock \doi{10.1561/0400000060}.
\newblock URL \url{https://doi.org/10.1561/0400000060}.

\end{thebibliography}
\newpage
\appendix

\section{Additional results}\label{sec:results}

\begin{definition}[$\epsilon$-subspace embedding \cite{10.1561/0400000060}]
    An $\epsilon$-subspace embedding for a matrix $B \in \R^{d \times k}$ is a matrix $S \in \R^{l \times d}$ such that
    \begin{equation}
        (1-\epsilon)\normTwo{y}^{2} \leq \normTwo{Sy}^{2} \leq (1+\epsilon)\normTwo{y}^{2} \texteq{for all $y \in Y = \{y :y = Bz, z\in \R^{k}\}$}.
    \end{equation}
\end{definition}

\begin{definition}[Oblivious subspace embedding \cite{10.1561/0400000060, 10.1109/FOCS.2006.37}]
    A distribution $\mathcal{S}$ on $S \in \R^{l \times d}$ is an $(\epsilon, \delta)$-oblivious subspace embedding for a given fixed/arbitrary matrix $B \in \R^{d \times k}$ , we have that, with a high probability of at least $1-\delta$, a matrix $S$ from the distribution is an $\varepsilon$-subspace embedding for $B$.
    \label{def:oblivious_embedding}
\end{definition}

\begin{lemma}[Theorem 2.3 in \cite{10.1561/0400000060}] \label{lem:Gauss_embedding}
Let $\epsilon_S \in (0,1)$ and $S \in \R^{l \times d}$ 
be a scaled Gaussian matrix. 
Then for any fixed $d \times (d+1)$ matrix $M$ with rank at most $r+1$, 
with probability $1-\delta_S$ we have that simultaneously for all 
$z \in \R^{d+1}$, $\normTwo{SMz}^2 \geq (1-\epsilon_S) \normTwo{Mz}^2$, 
where

\begin{equation}
    \delta_S = \exp{-\frac{l(\epsilon_S)^{2}}{C_{l}} + r + 1} 
\end{equation}
and $C_l>0$ is an absolute constant. 
\end{lemma}

\subsection{Proof of Lemma \ref{lem:low:rank:hessian}}
In order to prove Lemma \ref{lem:low:rank:hessian}, we need the following result.
\begin{lemma}[\cite{cartis_learning_2024, cosson_gradient_2022}]\label{lem:low:rank:characterisation}
    A function $f:\R^{d} \rightarrow \R$ has rank $r$, with $r \leq d$ if and only if there exists a matrix $A \in \R^{r \times d}$ and a map $\sigma:\R^{r} \rightarrow \R$ such that $f(x) = \sigma(Ax)$ for all $x \in \R^{d}$.
\end{lemma}

\begin{proof}[Proof of Lemma \ref{lem:low:rank:hessian}]
Using Lemma \ref{lem:low:rank:characterisation}, we can write
\begin{equation}
    f(x) = \sigma(Ax)
\end{equation}
where $A \in \mathbb{R}^{r \times d}$ and clearly, $A$ has rank at most $r$. We have
\begin{equation}
    \nabla f(x) =
    \frac{\partial \sigma(Ax)}{\partial x} =
    A^{\top} \cdot \frac{\partial \sigma(Ax)}{\partial (Ax)} =
    A^{\top} \nabla \sigma(Ax).
\end{equation}
Furthermore, as $f$ is $C^2$ by assumption, we have

\begin{equation}
    \nabla^{2} f(x) =
    \frac{\partial \nabla f(x)}{\partial x^T} =
    A^{\top} \cdot  \frac{\partial \nabla\sigma(Ax)}{\partial x^T} =
    A^{\top} \cdot \frac{\partial \nabla\sigma(Ax)}{\partial (Ax)^T} \cdot A = A^{\top}[\nabla^{2}\sigma(Ax)]A.
\end{equation}
As ${\rm rank}(A) \leq r$, we can conclude that $\nabla^{2} f(x)$ has rank bounded above by $r$.
\end{proof}

\section{Data profile methodology}
\label{sec:data_profiles}

We measure algorithm performance using data profiles \cite{moreBenchmarkingDerivativeFreeOptimization2009}, which themselves are a variant of performance profiles \cite{dolan2002benchmarking}. We use \textit{relative Hessians seen} as well as runtime for our data profiles. The relative Hessians seen at an iteration $k$ is $(l_{k}/d)^{2}$ where $l_{k}$ is the sketch dimension and $d$ is the problem dimension.
Using  the same notation as in \cite{cartis_randomised_2022}, for a given solver $s$, test problem $p \in \mathcal{P}$ and tolerance $\tau \in (0, 1)$, we determine the number of relative Hessians seen $N_{p}(s, \tau)$ required for a problem to be solved:
\begin{equation*}
    N_{p}(s, \tau) :=\ \text{\# of relative Hessians seen until}\ f(x_{k}) \leq f(x^{*}) + \tau (f(x_{0}) - f(x^{*})).
\end{equation*}
We set $N_{p}(s, \tau) = \infty$ in the cases where the tolerance is not reached within the maximum number of iterations, which we take to be 2000.

To produce the data profiles, we plot
\begin{equation*}
    \pi_{s, \tau}^{N}(\alpha):= \frac{|\{p \in \mathcal{P}\ :\ N_{p}(s, \tau) \leq \alpha\}|}{|\mathcal{P}|}\ \text{for}\ \alpha \in [0, 100],
\end{equation*}
namely, the fraction of problems solved after $\alpha$ relative Hessians seen. For runtime data profiles, we replace relative Hessians seen with the runtime in the above definitions.

\section{Low-rank problems}

{\footnotesize{\begin{table}[H]
\centering
\begin{tabular}{c l c c c c c}  
\toprule
\# & Problem & $d$ & $r$ & $f(x_{0})$ & $f(x^{*})$ & Parameters \\
\midrule
1 & l-ARTIF & 1000 & 100 & $1.8296\times 10^{1}$ & $0$ & N = 100 \\
2 & l-ARWHEAD & 1000 & 100 & $2.9700\times 10^{2}$ & $0$ & N = 100 \\
3 & l-BDEXP & 1000 & 100 & $2.6526\times 10^{1}$ & $0$ & N = 100 \\
4 & l-BOX & 1000 & 100 & 0 & $-1.1240\times 10^{1}$ & N = 100 \\
5 & l-BOXPOWER & 1000 & 100 & $8.6625\times 10^{2}$ & $0$ & N = 100 \\
6 & l-BROYDN7D & 1000 & 100 & $3.5098\times 10^{2}$ & $4.0123\times 10^{1}$ & N/2 = 50 \\
7 & l-CHARDIS1 & 1000 & 98 & $1.2817\times 10^{1}$ & $0$ & NP1 = 50 \\
8 & l-COSINE & 1000 & 100 & $8.6881\times 10^{1}$ & $-9.9000\times 10^{1}$ & N = 100 \\
9 & l-CURLY10 & 1000 & 100 & $-6.2372\times 10^{-3}$ & $-1.0032\times 10^{4}$ & N = 100 \\
10 & l-CURLY20 & 1000 & 100 & $-1.2965\times 10^{-2}$ & $-1.0032\times 10^{4}$ & N = 100 \\
11 & l-DIXMAANA1 & 1000 & 90 & $8.5600\times 10^{2}$ & 1 & M = 30 \\
12 & l-DIXMAANF & 1000 & 90 & $1.2253\times 10^{3}$ & 1 & M = 30 \\
13 & l-DIXMAANP & 1000 & 90 & $2.1286\times 10^{3}$ & 1 & M = 30 \\
14 & l-ENGVAL1 & 1000 & 100 & $5.8410\times 10^{3}$ & 0 & N = 100 \\
15 & l-FMINSRF2 & 1000 & 121 & $2.5075\times 10^{1}$ & 1 & P = 11 \\
16 & l-FMINSURF & 1000 & 121 & $3.0430\times 10^{1}$ & 1 & P = 11 \\
17 & l-NCB20 & 1000 & 110 & $2.0200\times 10^{2}$ & $1.7974\times 10^{2}$ & N = 100 \\
18 & l-NCB20B & 1000 & 100 & $2\times 10^{2}$ & $1.9668\times 10^{2}$ & N = 100 \\
19 & l-NONCVXU2 & 1000 & 100 & $2.6397\times 10^{6}$ & $2.3168\times 10^{2}$ & N = 100 \\
20 & l-NONCVXUN & 1000 & 100 & $2.7270\times 10^{6}$ & $2.3168\times 10^{2}$ & N = 100 \\
21 & l-NONDQUAR & 1000 & 100 & $1.0600\times 10^{2}$ & 0 & N = 100 \\
22 & l-ODC & 1000 & 100 & 0 & $-1.9802\times 10^{-2}$ & (NX, NY) = (10, 10) \\
23 & l-OSCIGRNE & 1000 & 100 & $3.0604\times 10^{8}$ & 0 & N = 100 \\
24 & l-PENALTY3 & 1000 & 100 & $9.8018\times 10^{7}$ & $1\times 10^{-3}$ & N/2 = 50 \\
25 & l-POWER & 1000 & 100 & $2.5503\times 10^{7}$ & 0 & N = 100 \\
26 & l-RAYBENDL & 1000 & 126 & $9.8028\times 10^{1}$ & $9.6242\times 10^{1}$ & NKNOTS = 64 \\
27 & l-SCHMVETT & 1000 & 100 & $-2.8029\times 10^{2}$ & $-2.9940\times 10^{3}$ & N = 100 \\
28 & l-SINEALI & 1000 & 100 & $-8.4147\times 10^{-1}$ & $-9.9010\times 10^{3}$ & N = 100 \\
29 & l-SINQUAD & 1000 & 100 & $6.5610\times 10^{-1}$ & $-3$ & N = 100 \\
30 & l-TOINTGSS & 1000 & 100 & $8.9200\times 10^{2}$ & $1.0102\times 10^{1}$ & N = 100 \\
31 & l-YATP2SQ & 1000 & 120 & $9.1584\times 10^{4}$ & 0 & N = 10 \\
\bottomrule
\end{tabular}
\caption{The 31 low-rank CUTEst test problems used in the data profiles.}\label{tab:cutest_lowrank}
\end{table}}}

\section{Individual problem performance}
\label{sec:problem_by_problem}
\begin{figure}[H]
    \centering
    \includegraphics[width=0.4\linewidth]{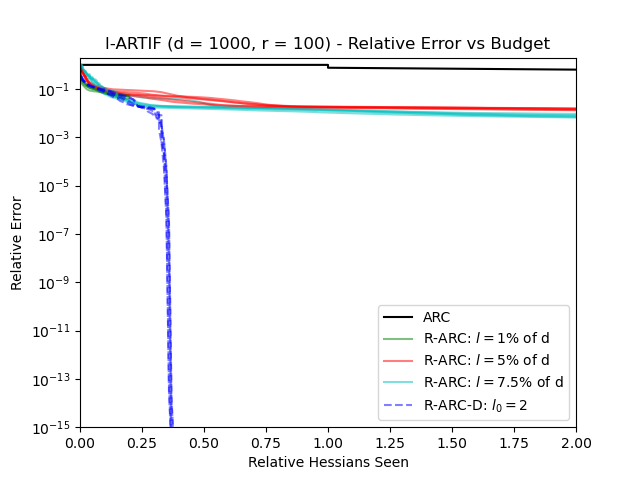}
    \includegraphics[width=0.4\linewidth]{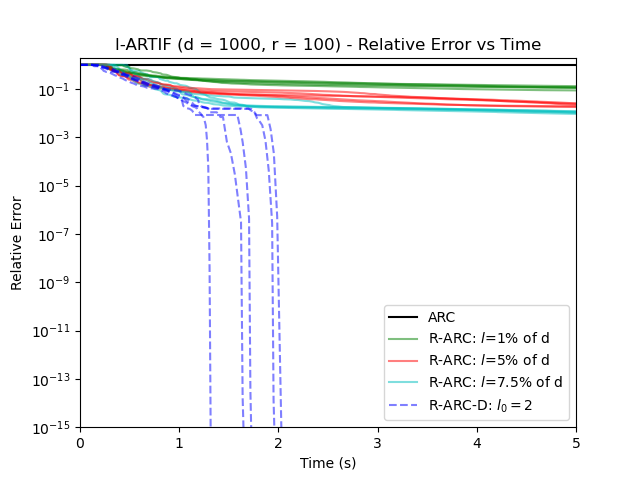}
    \includegraphics[width=0.4\linewidth]{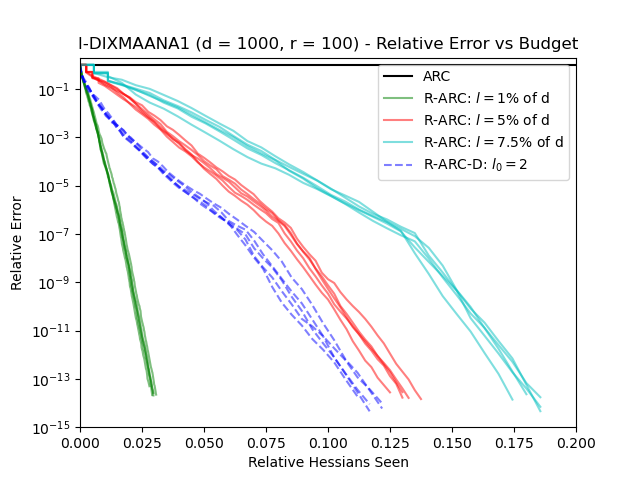}
    \includegraphics[width=0.4\linewidth]{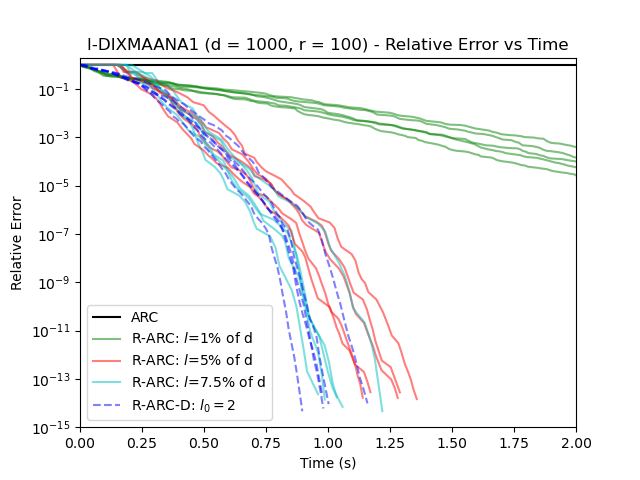}
    \includegraphics[width=0.4\linewidth]{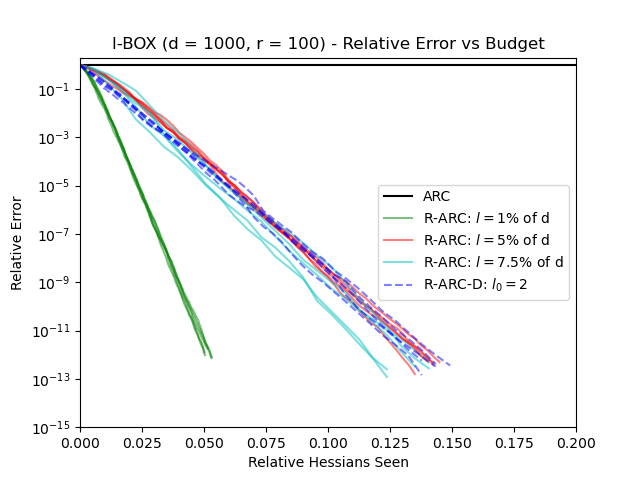}
    \includegraphics[width=0.4\linewidth]{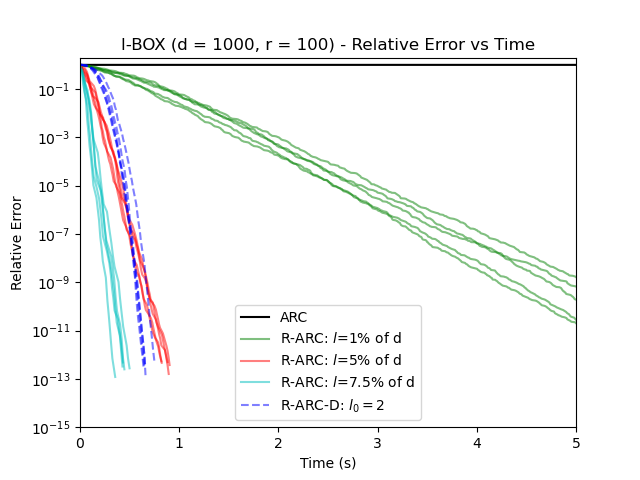}
    \includegraphics[width=0.4\linewidth]{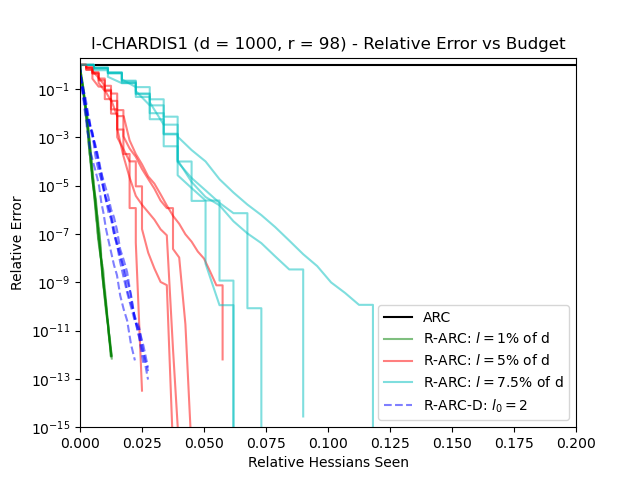}
    \includegraphics[width=0.4\linewidth]{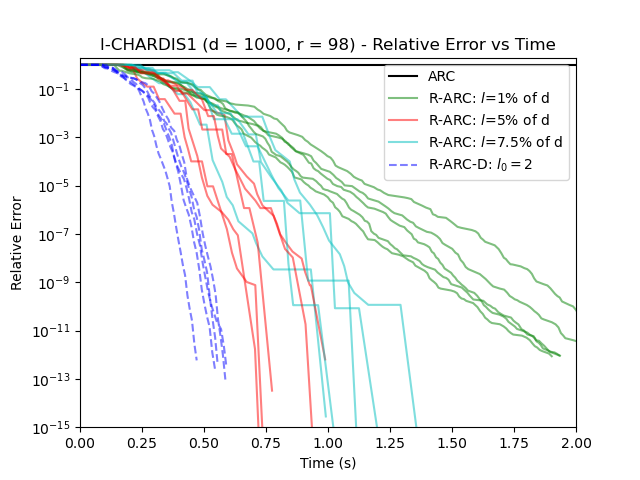}
    \label{fig:individual_problems}
    \caption{Comparison between R-ARC-D and R-ARC on low-rank problems from Table \ref{tab:cutest_lowrank}}
\end{figure}
\begin{figure}[H]
    \centering
    \includegraphics[width=0.48\linewidth]{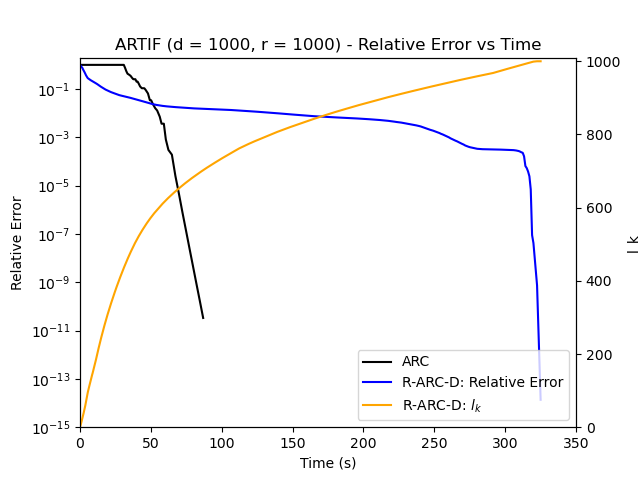}
    \includegraphics[width=0.48\linewidth]{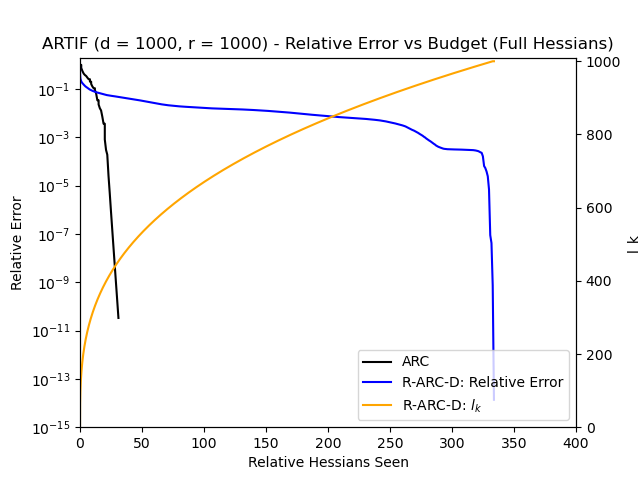}
    \caption{Example of R-ARC-D applied to the full-rank problem ARTIF (with parameter N = 1000), which has $r = d = 1000$.}
    \label{fig:R_ARC_D_update_ARTIF}
\end{figure}

\end{document}